\documentclass[11pt]{amsart}
\usepackage{amssymb,latexsym,amsmath}
\usepackage{graphics}                 
\usepackage{color}                    
\usepackage{hyperref}                 
\usepackage[all]{xy}
\usepackage[textheight=615pt, textwidth=360pt]{geometry}
\begin{document}

\newtheorem{definition}{Definition}
\newtheorem{theorem}{Theorem}[section]
\newtheorem{proposition}[theorem]{Proposition}
\newtheorem{lemma}[theorem]{Lemma}
\newtheorem{corollary}[theorem]{Corollary}
\newtheorem{question}[theorem]{Question}
\newtheorem{remark}[theorem]{Remark}
\newtheorem{example}[theorem]{Example}
\newtheorem{conjecture}[theorem]{Conjecture}

\newtheorem{correction}{Correction}

\def\A{{\mathbb{A}}}
\def\C{{\mathbb{C}}}

\def\L{{\mathbb{L}}}
\def\O{{\mathcal{O}}}
\def\P{{\mathbb{P}}}
\def\Q{{\mathbb{Q}}}
\def\Z{{\mathbb{Z}}}
\def\Ch{{\rm Ch}}
\def\p{{\mathbf{p}}}

\def\sp{{\rm SP}}
\def\rank{{\rm rank}}

\def\mC{{\mathcal{C}}}
\def\cZ{{\mathcal{Z}}}
\def\D{{\mathcal{D}}}

\title{Chow groups of Chow varieties }

\date{\today}




\author[Y. Chen]{Youming Chen }
\address{School of Science, Chongqing University of Technology, Chongqing 400054, P.R. China}%
\email{youmingchen@cqut.edu.cn}%

\author[W. Hu]{Wenchuan Hu}
\address{School of Mathematics, Sichuan University, Chengdu 610064, P.R. China}%
\email{huwenchuan@gmail.com}%


\begin{abstract}
 Let $C_{p,d}(\mathbb{P}^n)$ be the Chow variety  of effective algebraic $p$-cycles of degree $d$ in complex projective $n$-space $\mathbb{P}^n$. In this paper, we compute the rational Chow groups  $\mathrm{Ch}_q(C_{p,d}(\mathbb{P}^n))_\mathbb{Q}$ for $0 \le q \le d$. We show that these Chow groups are isomorphic to the corresponding rational singular homology groups $H_{2q}(C_{p,d}(\mathbb{P}^n), \mathbb{Q})$ in this range, a result that was previously known.  Furthermore, we prove that the rational Chow groups of a natural completion of the Chow monoid of algebraic $p$-cycles on projective spaces coincide with the corresponding rational singular homology groups. We also establish the stability of Chow groups of Chow varieties under natural embeddings and algebraic suspension maps within a certain range. Finally, we determine the Chow groups, up to a certain level, for the space of algebraic cycles of fixed degree.
\end{abstract}

\maketitle

\tableofcontents

\section{Introduction}

Let $\P^n$ be the  complex projective space of dimension $n$.
The Chow variety $C_{p,d}(\P^n)$
is the space of effective algebraic $p$-cycles of degree $d$ on $\P^n$.
As the moduli space of effective algebraic cycles
of a projective algebraic variety of suitable dimension and degree,
the importance of understanding the structure of Chow varieties is obvious.
It relates to many deep problems in  algebraic geometry such as the Hodge Conjecture.
The Chow variety $C_{p,d}(\P^n)$ plays a key role in the theory of  algebraic cocycles
or the morphic cohomology theory introduced by Friedlander and Lawson (see \cite{FL}).
Conjectured by Grothendieck(see \cite{Grothendieck})  and Friedlander-Mazur (see \cite[Chap. 7]{Friedlander-Mazur}), the filtration on the singular cohomology in rational coefficients  of a smooth complex projective  variety
given by the cycle class map of the morphic cohomology groups
 coincides with the corresponding Hodge filtration with  the maximal sub-Hodge structure.

Over the past decades,
much serious thought has been devoted to understanding the structure of Chow varieties.
The first important fact proved by Chow and Van der Waerden is that
$C_{p,d}(\P^n)$ canonically  carries the structure of a closed complex projective algebraic set (see \cite{Chow-Waerden}, \cite{GKZ}, \cite{Kollar}).
Therefore, it naturally admits the structure of a compact Hausdorff space.
Naturally, the most interesting structure of $C_{p,d}(\P^n)$ lies in topological and algebraic aspects.

On the topological aspect,
the topological structure of $C_{p,d}(\P^n)$ for
large $p$, $n$ and $d$  is  quite complicated.
For example, it is nontrivial to identify
the dimension of $C_{p,d}(\P^n)$ (see \cite{Eisenbud-Harris}, \cite{Lehmann}).
One can see easily that $C_{p,d}(\P^n)$ is a path-connected topology space,
but the proof of its simply connectedness (see \cite{Horrocks,Fujiki,Lawson1}),
or the computations of other topological invariants such as higher homotopy groups,
(singular) homology groups (see \cite{Lawson1,Hu-Deg2,Hu-2015}),
and Euler characteristic (see \cite{Lawson-Yau, Hu-2013}), etc.,
are quite difficult.
The complete identification of  higher homotopy groups or
homology groups for $C_{p,d}(\P^n)$  is still open.
A short survey  of the structure of Chow varieties and open questions can be found in \cite{Hu-2021}
and references therein.

On the algebraic  aspect,  the algebraic structure of $C_{p,d}(\P^n)$ as a projective algebraic set
seems to be much more complicated.
In general, $C_{p,d}(\P^n)$ is singular and non-irreducible,
and  even the number of irreducible components is unknown for $p=1,n=3$ and $d$ large.
Shafarevich once asked a question that whether each
irreducible component of $C_{p,d}(\P^n)$  is rational (see \cite{Shafarevich}) and the answer is negative (see \cite{Hu-2021}).

Due to the complexity of the algebraic structure of Chow varieties,
very little progress has been made in this aspect.
As one of the most important  algebraic invariants of an algebraic variety or scheme,
the structure of its Chow groups is the main target in the algebraic cycles theory.
In this direction, the Chow group $\Ch_0(C_{p,d}(\P^n))$ of zero cycles  has been identified
to be the group of integers  for all $d > 0$ and $0\leq p\leq n$ (see  \cite{Hu-2021}).
Furthermore, it has been wildly conjectured that
\begin{conjecture}[{\cite[Conj.4.9]{Hu-2021}}]\label{conj1}
For all $d\geq 0$ and $0\leq p\leq n$,
$$
\Ch_q(C_{p,d}(\P^n))\cong H_{2q}(C_{p,d}(\P^n),\Z),\ \forall q\geq 0,
$$
where $H_{*}(-,\Z)$ denotes the singular homology group in integer coefficients.
\end{conjecture}


Note that $C_{0,d}(\P^n)$ is the $d$-th symmetric product $\sp^d(\P^n)$ of $\P^n$.
For $p=0$,
whether the special case in Conjecture \ref{conj1} holds or not is of fundamental importance
to understand the structure of the symmetric product of projective spaces.
Hence it has independent interest and we formulate it as a separate conjecture.
\begin{conjecture}\label{conj2}
For all $d\geq 0$, $q\geq 0$ and $0\leq p\leq n$, the cycle class map
\begin{equation}\label{eqn1.1}
cl:\Ch_q(\sp^d(\P^n))\stackrel{\cong}{\longrightarrow} H_{2q}(\sp^d(\P^n),\Z).
\end{equation}
is an isomorphism.
\end{conjecture}

It might be too optimistic to hope Conjecture \ref{conj1} holds in integer coefficients, even for the Chow variety parameterizing zero cycles.
But there is no evidence that the cycle class map $cl$ is not injective or surjective.
However, it is more probable that Conjecture \ref{conj2}
holds for $0\leq q \leq d$.
Furthermore, it is not hard to show that Equation \eqref{eqn1.1}
holds after tensoring rational coefficients (see Lemma \ref{lemma3.6}).
Naturally, we propose a weaker version of Conjecture \ref{conj1} as follows.
\begin{conjecture}\label{conj3}
For all $d\geq 0$ and $0\leq p\leq n$,
$$
\Ch_q(C_{p,d}(\P^n))_{\Q}\cong H_{2q}(C_{p,d}(\P^n),\Q),\ \forall q\geq 0.
$$
Here $H_{*}(-,\Q)$ denotes the singular homology group in rational coefficients.
\end{conjecture}
This conjecture implies that all topological cycles of even dimensional (in rational coefficients) on Chow varieties are  algebraic. We also conjecture that  all odd Betti numbers of
$C_{p,d}(\P^n)$ vanish.

Recall that for a complex projective variety or scheme  $X$,
the Chow group of $q$-cycles $\Ch_q(X)$ on $X$ is defined to be
$$
\Ch_q(X)=\cZ_q(X)/{\rm \{rational ~equivalence\}},
$$
where $\cZ_q(X)$  is the group of algebraic $q$-cycles on $X$.
Set $\Ch_q(X)_{\Q}:=\Ch_q(X)\otimes \Q$. There are natural cycle class maps
$cl:\Ch_q(X)\to H_{2q}(X,\Z)$  and $cl_{\Q}:\Ch_q(X)_{\Q}\to H_{2q}(X,\Q)$.
For more details on the theory of Chow groups, the reader is referred to Fulton (\cite{Fulton}).

The Chow groups carry rich information of a variety or scheme,
and they are in general  extremely hard to compute.
For most varieties, the structure of  their Chow groups
is far from being understood.
The most deepest conjectures,
such as  the Hodge conjecture and the Bloch-Beilinson conjecture in algebraic geometry,
are attempts to understand Chow groups.

In this paper, we identify Chow groups  $\Ch_q(C_{p,d}(\P^n))_{\Q}$ with their
corresponding singular homology groups in rational coefficients for $0\leq q\leq d$,
and this verified the wild conjecture \ref{conj3} in a certain range.
In particular,
we compute $\Ch_q(C_{p,d}(\P^n))_{\Q}$ for $0\leq q\leq d$
since the singular homology groups in this range have been obtained earlier (see \cite{Hu-2015}).
Moreover, we also prove the stability of Chow groups of Chow varieties
under natural embeddings or algebraic suspension map  in a certain range.
The precise statements are given in the next section.

\thanks{\emph{Acknowledgements}: }
This work is partially supported by the National Nature Science Foundation of China
(No. 12126309, 12126354, 12171351).
\section{Main results}







The notations in this paper are extensively borrowed from Lawson's paper (see \cite{Lawson1}).
Let $l_0\subset \P^n$ be a fixed $p$-dimensional linear subspace.
For each $d\geq 1$,
we consider the analytic embedding
\begin{equation}\label{eqn2.1}
i:C_{p,d}(\P^n)\hookrightarrow C_{p,d+1}(\P^n)
\end{equation}
defined by $c\mapsto c+l_0$.
From this sequence of embeddings, we can form the union
$$
\mC_{p}(\P^n)=\lim_{d\to\infty} C_{p,d}(\P^n).
$$
which is called the Friedlander completion of  $C_{p,d}(\P^n)$.
The topology on $\mC_{p}(\P^n)$ is the weak topology for $\{ C_{p,d}(\P^n)\}_{d=1}^{\infty}$,
that is,
a set $C\subset \mC_{p}(\P^n)$ is closed
if and only if $C\cap  C_{p,d}(\P^n)$ is closed in $C_{p,d}(\P^n)$ for all $d\geq 1$.
For more details, the reader is referred to the paper \cite{Lawson1}.

We define the Chow group $\Ch_q(\mC_p(\P^n))$ of algebraic $q$-cycles on $\mC_p(\P^n)$
to be the direct limit $\lim\limits_{d\to \infty}\Ch_q(C_{p,d}(\P^n))$.
Meanwhile, we have
$$
\lim_{d\to \infty}H_{2q}(C_{p,d}(\P^n),\Q)=H_{2q}(\lim_{d\to \infty}C_{p,d}(\P^n),\Q)
=H_{2q}(\mC_{p}(\P^n),\Q)
$$
for all $0\leq p\leq n$ and $q\geq 0$.

Our first main result states as follows.

\begin{theorem} \label{Thm2.1}
For $0\leq p\leq n$ and $q\geq 0$, we have isomorphisms
$$
\Ch_q(\mC_p(\P^n))_{\Q}\cong H_{2q}(\mC_p(\P^n),\Q).
$$
\end{theorem}

This verifies Conjecture \ref{conj3} for the Friedlander completion of Chow varieties.
 Theorem \ref{Thm2.1} indicates that every  topological cycle in rational coefficients on
Chow varieties $C_{p,d}(\P^n)$ are homologous to an algebraic cycle under the direct limit
of embeddings \ref{eqn2.1}.

The following result is a Chow group version of a conjecture by Lawson,
i.e., $i:C_{p,d}(\P^n)\to C_{p,d+1}(\P^n)$ is $2d$-connected,
which follows from the main result in \cite{Hu-2015}.
\begin{theorem}\label{Thm2.2}
For all $0\leq p\leq n$,
the embedding  $i:C_{p,d}(\P^n)\to C_{p,d+1}(\P^n)$ induces
isomorphism $i_*:\Ch_q(C_{p,d}(\P^n))_{\Q}\cong \Ch_q(C_{p,d+1}(\P^n))_{\Q}$ for $0\leq q\leq d$.
\end{theorem}

Combining results in the above two theorems and the homotopy type of $\mC_p(\P^n)$,
we identify the structure of the $q$-th Chow group  of $C_{p,d}(\P^n)$ as
$$
\Ch_q(C_{p,d}(\P^n))_{\Q}\cong H_{2q}(\mC_p(\P^n),\Q)
$$
for all  $0\leq q\leq d$, where
$$
H_{2q}(\mC_p(\P^n),\Q)\cong H_{2q}\bigg(\prod_{i=1}^{n-p} K(\Z,2i),\Q\bigg),
$$
and $K(\Z,m)$ is the Eilenberg-MacLane space
whose $m$-th homotopy group is isomorphic to $\Z$,
and all other homotopy groups are trivial.
Note that the $\Q$-dimension of
$H_{2q}\big(\prod_{i=1}^{n-p} K(\Z,2i),\Q\big)$ can be easily computed by the K\"{u}nneth formula.

The following result is the Chow group version of the Lawson suspension theorem for homotopy groups.
\begin{theorem}\label{Thm2.3}
For all $0\leq p\leq n$,
the suspension map $\Sigma:C_{p,d}(\P^n)\to C_{p+1,d}(\P^{n+1})$ induces
isomorphism
$$\Sigma_*:\Ch_q(C_{p,d}(\P^n))_{\Q}\cong \Ch_q(C_{p+1,d}(\P^{n+1}))_{\Q}$$ for $0\leq q\leq d$.
\end{theorem}

The suspension map $\Sigma$ is defined in the next section.

We will prove these theorems by developing Lawson's idea
in studying homotopy type of Chow monoids \cite{Lawson1}.
Compared to Lawson's method on homotopy groups,
there are significant differences and changes in the proof of our results.
One of the big differences is that Chow groups are not homotopy invariants.



As applications of Theorem \ref{Thm2.1} and \ref{Thm2.2},
the Lawson suspension theorem \cite[Th.3]{Lawson1} and \cite[Th.1]{Hu-2015},
we obtain Chow groups of $C_{p,d}(\P^n)$ up to degree $d$,
and this partly verify Conjecture \ref{conj3}.

 \begin{corollary}\label{cor2.4}
For all $0\leq p\leq n$, the cycle class map
$$
cl:\Ch_q(C_{p,d}(\P^n))_{\Q}\cong H_{2q}(C_{p,d}(\P^n),\Q)
$$
is an isomorphism for any $0\leq q\leq d$.
\end{corollary}

The first $2d$ homology group of $C_{p,d}(\P^n)$
has been computed by the formula in \cite[Cor. 5]{Hu-2015},
hence by Corollary \ref{cor2.4} we obtain the rank of $\Ch_q(C_{p,d}(\P^n))$ for $0\leq q\leq d$ and all $0\leq p\leq n$.

\begin{remark}\label{rmk2.5}
If the Conjecture \ref{conj2} holds, or
its special case that
$$cl:\Ch_q(\sp^d(\P^n))\stackrel{\cong}{\longrightarrow} H_{2q}(\sp^d(\P^n),\Z),\  0\leq q\leq d$$
  holds,
then Corollary \ref{cor2.4} holds in the same range in integer coefficients.
Moreover,
if the cycle class map $cl$ on $\Ch_q(\sp^d(\P^n))$ is injective (resp. surjective),
then the cycle class map on $\Ch_q(C_{p,d}(\P^n))$ is also injective (resp. surjective).
These statements are shown in the proof of Theorem \ref{Thm2.1}.
\end{remark}

In particular, for $0$-cycles,
we obtain $\Ch_0(C_{p,d}(\P^n))\cong \Z$ for all $0\leq p\leq n$ and
$d>0$ from Corollary \ref{cor2.4} and Remark \ref{rmk2.5},
since by the rationality of $\sp^d(\P^n)$, $\Ch_0(\sp^d(\P^n))\cong \Z$.
This recovers a result in \cite[Prop. 4.7]{Hu-2021},
where a $\C^*$-action on the Chow variety $C_{p,d}(\P^n)$ is involved.

For $1$-cycles,
we have $\Ch_1(C_{p,d}(\P^n))\cong\Z\oplus Tor_1$
and $H_{2}(C_{p,d}(\P^n))\cong \Z$
for all $0\leq p< n$ and $d>0$,
where $Tor_1$ is the torsion group of $\Ch_1(C_{p,d}(\P^n))$.
It is probably that
$Tor_1$ is actually zero for all $p,d,n$.
This is the Chow group  analogue
of a result for the homotopy group of $C_{p,d}(\P^n)$ (see \cite[Cor. 3.7]{Hu-Deg2}).

\begin{remark}
The analogous result of  Theorem \ref{Thm2.1}, \ref{Thm2.2}
for Chow varieties over algebraically closed fields
can be obtained by applying Friedlander's generalization
of Complex Suspension method to algebraic cases,
where singular homology groups will be replaced by $\ell$-adic homology groups.
\end{remark}

\section{The technique of Suspension on Chow groups}

Now we briefly review Lawson's idea in the proof of the Complex Suspension Theorem.
More detailed explanation of the Suspension Theorem over the complex number field,
or arbitrary algebraically closed field,
and background materials in this section
can be found in \cite{Lawson1}, \cite{Friedlander1} and \cite{Lawson2}.
Inspired by Lawson's work on homotopy type of the space of algebraic cycles, we
compute Chow groups of Chow varieties and its Friedlander completion.

Fix a hyperplane $\P^n\subset\P^{n+1}$ and a point $\P^0\in \P^{n+1}-\P^{n}$.
Let $V\subset \P^n$ be any closed algebraic subset.
The \textbf{algebraic suspension of $V$} with vertex $\P^0$ (i.e., cone over $V$) is the set
$$
\Sigma V:=\cup\{ l~|~ l \hbox{ is a projective line through $\P^0$ and intersects $V$}\}.
$$
The suspension map $\Sigma$ induces a morphism $\Sigma: C_{p,d}(\P^{n}) \to C_{p+1,d}(\P^{n+1})$
and it commutes with the embedding $i$ in Equation \eqref{eqn2.1}.
Hence there is an induced map $\Sigma:\mC_{p}(\P^{n}) \to \mC_{p+1}(\P^{n+1})$.
Moreover, $\Sigma$ induces homomorphisms on Chow groups
$\Sigma_*: \Ch_q(C_{p,d}(\P^{n})) \to \Ch_q(C_{p+1,d}(\P^{n+1}))$ and
$\Sigma_*: \Ch_q(\mC_{p}(\P^{n})) \to \Ch_q(\mC_{p+1}(\P^{n+1}))$ for all $q\geq 0$.

\begin{theorem}\label{Thm3.1}
For all $q\geq 0$,
the induced homomorphisms
$$\Sigma_*:\Ch_q(\mC_{p}(\P^{n})) \to \Ch_q(\mC_{p+1}(\P^{n+1}))$$
are surjective, and they are isomorphic after tensoring with rational coefficients.
\end{theorem}

This is the Chow group version of the Lawson Suspension Theorem on homotopy and homology groups.
We will show this theorem at the end of this section
since we need some preparation.

Consider a holomorphic vector field which is zero on $\P^n$ and on its polar point $\P^0$
as follows:
Choose homogenous coordinates $[z_0:z_1:...:z_{n+1}]$ for $\P^{n+1}$ such that
$\P^n=(z_0=0)$ and $\P^0=[1:0:...:0]$.
For $t\in \C^*$, set
$$
\phi_t([z_0:z_1:...:z_{n+1}])=[tz_0:z_1:...:z_{n+1}].
$$
For $t\in \C^*$,
the map $\phi_t:\P^{n+1}\to \P^{n+1}$ is an automorphism,
and hence it induces an automorphism
$$
\phi_t: C_{p+1,d}(\P^{n+1}) \to C_{p+1,d}(\P^{n+1}).
$$
It is clear that $\phi_1$ is the identity map.
Moreover, $\phi_t$ preserves
the subspace $T_{p+1,d}(\P^{n+1})$ and it is identity on $\Sigma C_{p,d}(\P^{n}) $
for all $t\in \C^*$.
Here
$$
T_{p+1,d}(\P^{n+1}):=\big\{c=\sum n_iV_i\in C_{p+1,d}(\P^{n+1})|\dim(V_i\cap\P^n)=p, ~\forall i\big \}
$$
for any non-negative integer $p$ and $d$.
(When $d=0$, $C_{p,0}(\P^n)$ is defined to be the empty cycle.)

\begin{proposition}\label{Prop3.2}
The set $T_{p+1,d}(\P^{n+1})$ is Zariski open in $C_{p+1,d}(\P^{n+1})$.
Moreover, $T_{p+1,d}(\P^{n+1})$ is homotopy equivalent to $C_{p,d}(\P^{n})$.
For each $c\in T_{p+1,d}(\P^{n+1})$, the limit
$$
\phi_{\infty}(c):=\lim_{t\to \infty}\phi_t(c)\in \Sigma C_{p,d}(\P^{n})
$$
exists. Moreover,
there is an algebraic map
$$
\Phi: T_{p+1,d}(\P^{n+1})\times\C\to T_{p+1,d}(\P^{n+1})
$$
such that $\Phi_{|T_{p+1,d}(\P^{n+1})\times \{1\}}=\phi_1$
and $\Phi_{|T_{p+1,d}(\P^{n+1})\times \{\infty\}}=\phi_{\infty}$,
where $\C=\C^*\cup \{\infty\}$.
This implies that $\phi_1(Z)=Z$ is rationally equivalent to $\phi_{\infty}(Z)$
if $Z$ is an (effective) cycle on $C_{p+1,d}(\P^{n+1})$
supported in $T_{p+1,d}(\P^{n+1})$.
\end{proposition}
\begin{proof}
The essential idea of this proposition
can be traced back to Fulton (see \cite[Chap. 5]{Fulton} and \cite[Remark 4.6]{Lawson1}).
It has been shown \cite[\S 4]{Lawson1} that
$T_{p+1,d}(\P^{n+1})$  is Zariski open in $C_{p+1,d}(\P^{n+1})$
and $\phi_{\infty}(c)$ exists for each $c\in T_{p+1,d}(\P^{n+1})$.

Note that $\C=\C^*\cup \{\infty\}$, and we consider the morphism
$$
\Phi:C_{p+1,d}(\P^{n+1})\times \C^*\to C_{p+1,d}(\P^{n+1}),
$$
defined as $\Phi(c,t)=\phi_t(c)$.
The Zariski closure of the graph of $\Phi$
in $$(C_{p+1,d}(\P^{n+1})\times \C)\times C_{p+1,d}(\P^{n+1})$$
is denoted by $\overline{\Gamma_{\Phi}}$,
and the key point is that $\overline{\Gamma_{\Phi}}$ is single-valued
over $$T_{p+1,d}(\P^{n+1})\times \C\subset  C_{p+1,d}(\P^{n+1})\times \C$$ (see \cite[Rmk. 4.6]{Lawson1}).
This implies that if $Z$ is an (effective) algebraic cycle
on $C_{p+1,d}(\P^{n+1})$ supported in $T_{p+1,d}(\P^{n+1})$,
$\overline{\Gamma_{\Phi}}_{|(T_{p+1,d}(\P^{n+1})\times \C)}$ gives a rational equivalence between
$\phi_1(Z)=Z$  and $\phi_{\infty}(Z)$ whose support lies in $\Sigma C_{p,d}(\P^{n})$.
\end{proof}

\begin{remark}
For a projective algebraic variety $Z$  on $C_{p+1,d}(\P^{n+1})$
whose support is in $T_{p+1,d}(\P^{n+1})$,
we have the following morphisms
$$
Z\times\C^*\to T_{p+1,d}(\P^{n+1})\times \C^*\to C_{p+1,d}(\P^{n+1})\times \C^*\to C_{p+1,d}(\P^{n+1}).
$$
The morphism extends to a well-defined  algebraic map
(i.e., the graph is an algebraic subset in of the product) (see \cite[p.280]{Lawson1})
$$
\phi_{Z}:Z\times\C\to T_{p+1,d}(\P^{n+1})\times \C\to T_{p+1,d}(\P^{n+1}),
$$
whose image $W:=\phi_Z(Z\times\C)$,
where $\C=\C^*\cup \{\infty\}$.
The rational function field $K(W)$ of $W$ coincides with $K(Z\times \C^*)$.
Let $g\in K(W)$ be the element corresponding $t-1\in K(Z\times \C^*)$,
where $t$ is the local coordinate on $\C=\C^*\cup \{\infty\}$.
Then $div(g)=Z- \phi_{\infty}(Z)$.
\end{remark}

Now we study when an (effective) algebraic cycle $Z$ on $C_{p+1,d}(\P^{n+1})$,
under the sequence of embeddings as in Equation \eqref{eqn2.1},
can be rationally equivalent to an algebraic cycle $Z'$
in $C_{p+1,d'}(\P^{n+1})$ for large $d'$
such that the support of $Z'$ lies in $T_{p+1,d'}(\P^{n+1})$.

Fix a linear embedding $\P^{n+1}\subset \P^{n+2}$ and two points $x_0,x_1\in \P^{n+2}-\P^{n+1}$.
Each projection $p_i:\P^{n+2}-\{x_i\}\to \P^{n+1}$ ($i=0,1$)
gives us an algebraic line bundle over $\P^{n+1}$.

Let $D\in C_{n+1,e}(\P^{n+2})$ be an effective divisor of degree $e$ in $\P^{n+2}$
such that both $x_0, x_1$ are not in $D$.
Denote by $\widetilde{Div_e}(\P^{n+2})\subset C_{n+1,e}(\P^{n+2})$ the subset of all
such $D$.

Any effective $(p+1)$-cycle $c$ of degree $d$ on $\P^{n+1}$,
i.e.,
$c\in C_{p+1,d}(\P^{n+1})$,
 can be lifted to a cycle with support in $D$,
 defined as follows:
$$
\Psi_D(c)=(\Sigma_{x_0}c)\cdot D.
$$
The map $\Psi(c, D):=\Psi_D(c)$ is a continuous map with variables $c$ and $D$.
Hence we have a continuous map $\Psi_D:C_{p+1,d}(\P^{n+1})\to C_{p+1,de}(\P^{n+2}-\{x_0,x_1\})$,
which is in fact a morphism.
The composition of $\Psi_D$ with the projection $(p_0)_*$
is $(p_0)_*\circ\Psi_D=e$ (where $e\cdot c=c+\cdots+c$ for $e$ times).
The composed map of $\Psi_D$ with the projection $(p_1)_*$
gives us a transformation of cycles in $\P^{n+1}$
which makes most of them intersecting properly to $\P^{n}$.
To see this,
we consider the family of divisors $tD$, $t\in \C$,
given by  the scalar multiplication by $t$ in the line bundle
$p_0:\P^{n+2}-\{x_0\}\to \P^{n+1}$.

Assume $x_1$ is not in $tD$ for all $t\in \C$.
Then the above construction gives us a family transformation
$$
F_{tD}:=(p_1)_*\circ \Psi_{tD}: C_{p+1,d}(\P^{n+1})\to C_{p+1,de}(\P^{n+1})
$$
for $t\in \C$.
Note that $F_{0D}\equiv e$ (multiplication by $e$).

The question is that for a fixed $c$,
which divisors $D\in C_{n+1,e}(\P^{n+2})$
($x_0$ is not in $D$ and $x_1$ is not in $\bigcup_{0\leq t\leq 1} tD$) have the property that
$$
F_{tD}(c)\in T_{p+1,de}(\P^{n+1})
$$
for all $t\in \C^*$?
Set
$$
B_c:=\{D\in C_{n+1,e}(\P^{n+2})|F_{tD}(c) ~\hbox{is not in $T_{p+1,de} (\P^{n+1})$ for some  $t\in \C^*$}\},
$$
i.e., all divisors with degree $e$ on $\P^{n+2}$ such that some component of
\begin{equation*}
 (p_1)_*\circ \Psi_{tD}(c)\subset \P^n
\end{equation*}
for some $t\in \C^*$.
An important calculation we will use later is the following result.
\begin{proposition}[\cite{Lawson1}, Lemma 5.11]\label{Prop3.4}
For $c\in C_{p+1,d}(\P^{n+1})$, ${\rm codim}_{\C}B_c\geq \binom{p+e+1}{e}$.
\end{proposition}

We have the following result.
\begin{proposition}\label{Prop3.5}
Let $i_V:V\subset C_{p+1,d}(\P^{n+1})$ be a subvariety of dimension  $q$.  Then for each $e\in \Z\geq 0$
such that $ \binom{p+e+1}{e}>q+1$, there exists a divisor $D\in \widetilde{Div_e}(\P^{n+2})$ such that
$e\cdot V$ and $F_D (V)\subset T_{p+1,de}(\P^{n+1})$ are rationally equivalent as
algebraic cycles on $C_{p+1,de}(\P^{n+1})$.
\end{proposition}
\begin{proof}

For each $d, e>0$ and $p\leq n$,
there is a rational map
$$
\Psi:C_{p+1,d}(\P^{n+1})\times \widetilde{Div_e}(\P^{n+2})\to C_{p+1,de}(\P^{n+1})
$$
given by intersection (see \cite[Prop. 3.4-3.5]{Friedlander1}).

Now for any  integer $e$ satisfying $ \binom{p+e+1}{e}>q+1$, define the subset
$$
B(V)=\bigcup_{c\in V, t\in \C^*} t\cdot B_c,
$$
of $\widetilde{Div_e}(\P^{n+2})$,
where $ t\cdot B_c=\{tD|D\in B_c\}$.
By Proposition \ref{Prop3.4},
we have
$$
{\rm codim}_{\C} B(V)\geq  \binom{p+e+1}{e}-q-1>0.
$$
Hence there exists a $D\in \widetilde{Div_e}(\P^{n+2})-B(V)$ such that
$$
f:V\times \C\to C_{p+1,de}(\P^{n+1})
$$
is the restriction of the above $\Psi$, and it is defined everywhere on $V\times \C$.
This is to say,  $f(x,t)=F_{tD}(x)$.
This $f$ gives the rational equivalence
between $e\cdot V$ and $F_D(V)$ as elements in $C_{p+1,de}(\P^{n+1})$,
i.e., $e\cdot V$ is rationally equivalent to an element $F_D(V)$
supported in the open subset $T_{p+1,de}(\P^{n+1})$.
\end{proof}

We also need  the following result in the proof of Theorem \ref{Thm2.1}.
\begin{lemma}\label{lemma3.6}
For any $d>0$ and $0\leq p\leq n$,
the cycle class map
$$
cl_{\Q}:\Ch_q(\sp^d\P^{n})_{\Q}\to H_{2q}(\sp^d\P^{n},\Q)
$$
is an isomorphism.
\end{lemma}
\begin{proof}
For a projective variety $X$ with a finite group action $G$ such that $Y=X/G$,
we have $\Ch_q(Y)_{\Q}\cong \Ch_q(X)^{G}_{\Q}$ (see \cite[Example 1.7.6]{Fulton}),
and $H_{2q}(Y,{\Q})\cong H_{2q}(X,\Q)^{G}$.
If $\Ch_q(X)_{\Q}\cong H_{2q}(X,\Q)$,
then from the naturality of the cycle class map from Chow group to the singular homology group,
we get the isomorphism $\Ch_q(Y)_{\Q}\cong H_{2q}(Y,\Q)$.
Now applying $X=(\P^n)^d$ and the symmetric group $G=S_d$ of $d$ elements,
we obtain the Lemma.
\end{proof}


\begin{remark}
Lemma \ref{lemma3.6} is the basic case in our computation of the
Chow group of Chow varieties in rational coefficients.
Whether the cycle class map
$$
cl:\Ch_q(\sp^d\P^{n})\to H_{2q}(\sp^d\P^{n},\Z)
$$
is an isomorphism for all $d>0$ and $0\leq p\leq n$ is still open,
as the special case of conjecture \ref{conj1}.
Such a concise statement had been presumed ``trivially" true
until one was asked to give a proof.  It is interesting to show or disprove
this statement and it is worthwhile to  be formulated  as an independent conjecture.
\end{remark}

\begin{remark}
The Chow group (in rational coefficients) of infinite symmetric product of a smooth complete variety $X$
 have been defined in \cite{Kimura-Vistoli}.
The relations between the Chow group  of infinite symmetric product of $X$
and that of a finite symmetric product of $X$ are given by stability conjectures.
It seems that the only computable Chow groups of a finite symmetric product  is the case
that $X$ is a smooth projective curve.
\end{remark}

We need the following lemma to show Theorem \ref{Thm3.1}.
\begin{lemma}\label{lemma3.9}
For any $0\leq p'\leq n'$ and $q\geq 0$,
we have the following commutative diagram
\begin{equation}\label{eqn3.1}
\vcenter{
\xymatrix{&\Ch_q(\mC_{p'}(\P^{n'}))\ar[r]^-{cl}\ar@{->>}[d]^{\Sigma_*^c}&H_{2q}(\mC_{p'}(\P^{n'}),\Z)\ar[d]^{\Sigma_*^h}_{\cong}\\
&\Ch_q(\mC_{p'+1}(\P^{n'+1}))\ar[r]^-{cl}&H_{2q}(\mC_{p'+1}(\P^{n'+1}),\Z),
}
}
\end{equation}
where we write $\Sigma_*^c$ to denote the suspension map on Chow groups,
while $\Sigma_*^h$ for the suspension map on singular homology groups to avoid confusion.
Then the map $\Sigma_*^c:\Ch_q(\mC_{p'}(\P^{n'}))\to \Ch_q(\mC_{p'+1}(\P^{n'+1}))$  is surjective.
\end{lemma}
\begin{proof}
Note that we have the following commutative diagram
\begin{equation*}
\xymatrix{&\Ch_q(\mC_{p'}(\P^{n'}))\ar[r]^-{cl}\ar[d]^{\Sigma_*^c}&H_{2q}(\mC_{p'}(\P^{n'}),\Z)\ar[d]^{\Sigma_*^h}_{\cong}\\
&\Ch_q(\mC_{p'+1}(\P^{n'+1}))\ar[r]^-{cl}&H_{2q}(\mC_{p'+1}(\P^{n'+1}),\Z),
}
\end{equation*}
where the right vertical map is an isomorphism by the Lawson suspension theorem (see \cite{Lawson1}).

By definition of the Chow group $\Ch_q(\mC_{p'+1}(\P^{n'+1}))$,
an element $[Z]\in \Ch_q(\mC_{p'+1}(\P^{n'+1}))$
comes from the image of the class of an algebraic cycle $Z$
on $\Ch_q(C_{p'+1,d}(\P^{n'+1}))$ for $d$ large.
For any algebraic cycle, we can write $Z=Z^+-Z^-$,
where both $Z^+$ and $Z^-$ are effective algebraic cycles.
Hence it suffices to prove the lemma for the class of effective cycles.

Applied Proposition \ref{Prop3.5} to the effective $q$-cycle $Z$ on $C_{p'+1,d}(\P^{n'+1})$,
we obtain that $e\cdot Z$ is rationally equivalent to another effective algebraic cycle
$Z'$ in $C_{p'+1,de}(\P^{n'+1})$ for sufficiently large  positive integer $e$.
Moreover, the support of $Z'$ lies in $ T_{p'+1,de}(\P^{n'+1})$.
By Proposition \ref{Prop3.2},
$Z'$ is rationally equivalent to an effective algebraic cycle $Z''$ whose support
lies in $\Sigma C_{p',de}(\P^{n'})$.
Hence $e\cdot Z$ is rationally equivalent to $Z''$.
By taking the limit $d\to \infty$ on $\Sigma_*:\Ch_q(C_{p',d}(\P^{n'}))\to \Ch_q(C_{p'+1,d}(\P^{n'+1}))$,
this implies that $\Sigma_*:\Ch_q(\mC_{p'}(\P^{n'}))\to \Ch_q(\mC_{p'+1}(\P^{n'+1}))$ is surjective
and completes the proof of the lemma.
\end{proof}

Now we can prove Theorem \ref{Thm3.1} and Theorem \ref{Thm2.1}.

\begin{proof}[Proof of Proposition \ref{Thm3.1} and Theorem \ref{Thm2.1}]
Note that there is a natural cycle class map $cl:\Ch_q(X)\to H_{2q}(X,\Z)$ for any projective variety $X$.
In particular,
we have the natural cycle class map $cl: \Ch_q(C_{p,d}(\P^{n}))\to H_{2q}(C_{p,d}(\P^{n}),\Z)$.

When $p=0$, Theorem \ref{Thm2.1} follows from Lemma \ref{lemma3.6}.
For $p>0$, set $p=p'+1$.
Recall that we have the following commutative diagram
\begin{equation*}
\xymatrix{&\Ch_q(\mC_{p'}(\P^{n'}))\ar[r]^-{cl}\ar@{->>}[d]^{\Sigma_*^c}&H_{2q}(\mC_{p'}(\P^{n'}),\Z)\ar[d]^{\Sigma_*^h}_{\cong}\\
&\Ch_q(\mC_{p'+1}(\P^{n'+1}))\ar[r]^-{cl}&H_{2q}(\mC_{p'+1}(\P^{n'+1}),\Z)
}
\end{equation*}
for any $0\leq p'\leq n'$ and $q\geq 0$.
In Lemma \ref{lemma3.9}, we have shown that $\Sigma_*^c$ is surjective.
It follows from \cite[Thm.3]{Lawson1}
that $\Sigma_*^h$ is an isomorphism.
Therefore,
the upper horizontal cycle class map is surjective
if and only if the lower horizontal cycle class map is surjective.
Meanwhile,
if the upper horizontal cycle class map is injective,
then both $\Sigma_*^c$ and the lower horizontal cycle class map is injective.

After tensoring with rational coefficients in Equation \eqref{eqn3.1},
we have
\begin{equation}\label{eqn3.2}
\vcenter{
\xymatrix{&\Ch_q(\mC_{p'}(\P^{n'}))_{\Q}\ar[r]^-{cl_{\Q}}\ar@{->>}[d]^{\Sigma_*^c}
&H_{2q}(\mC_{p'}(\P^{n'}),\Q)\ar[d]^{\Sigma_*^h}_{\cong}\\
&\Ch_q(\mC_{p'+1}(\P^{n'+1}))_{\Q}\ar[r]^-{cl_{\Q}}&H_{2q}(\mC_{p'+1}(\P^{n'+1}),\Q).
}
}
\end{equation}
By Lemma \ref{lemma3.6},
we have the upper horizontal cycle class map in Equation \eqref{eqn3.2} is an isomorphism for $p'=0$.
Now by induction on $p'$,
we get $\Sigma_*^c$ and hence $cl_{\Q}$ are isomorphisms all $0\leq p\leq n$ and $q\geq 0$.
This completes the proof of Theorem \ref{Thm3.1} and Theorem \ref{Thm2.1}.
\end{proof}

\begin{remark}
It should be cautious about the fact that the Chow group is NOT a homotopy invariant.
Therefore the Chow group $\Ch_q(T_{p+1,d}(\P^{n+1}))$ is in general not isomorphic to
$\Ch_q(C_{p,d}(\P^{n}))$,
although $T_{p+1,d}(\P^{n+1})$ is homotopy equivalent to $C_{p,d}(\P^{n})$.
The Chow group $\Ch_q(T_{p+1,d}(\P^{n+1}))$ never appears in the proof of our results.
As a comparison, such a homotopy equivalence is essential in Lawson's proof of the suspension theorem.

Actually, in the proof of Theorem \ref{Thm2.1},  we first show the surjectivity of
$$
\Sigma_*^c:\Ch_q(\mC_{p}(\P^{n}))_{\Q}\to \Ch_q(\mC_{p+1}(\P^{n+1}))_{\Q},
$$
and then we prove the injectivity of $\Sigma_*^c$ by induction on $p$
and the corresponding isomorphism of singular homology groups,
while the Lawson suspension theorem is indispensable.
\end{remark}

\section{Proof of Theorem \ref{Thm2.2} and Theorem \ref{Thm2.3}}

Recall that in the above section,
the map $F_{tD}$ was
constructed,
i.e.,
$$
F_{tD}:=(p_1)_*\circ \Psi_{tD}: C_{p+1,d}(\P^{n+1})\to C_{p+1,de}(\P^{n+1}).
$$
Moreover,
the image of $F_{tD}$ is in the Zariski open subset $T_{p+1,de}(\P^{n+1})$ when $D$ is not in $B_c$.
We can find such a $D$ when ${\rm codim}_{\C}B_c\geq \binom{p+e+1}{e}$ is positive.

Let $i_V:V\subset C_{p+1,d}(\P^{n+1})$ be a subvariety of dimension $q$.
Then, by Proposition \ref{Prop3.5},
the cycle $e\cdot V=F_{0D}(V)$ is rationally equivalent
to the algebraic cycle $F_D(V)$ supported in $ T_{p+1,de}(\P^{n+1})$.
Therefore, we have the following commutative diagram
\begin{equation}\label{eqn4.1}
\vcenter{
\xymatrix{
&T_{p+1,d}(\P^{n+1})\ar[r]\ar@{^(->}@{^(->}[d]&T_{p+1,de}(\P^{n+1})\ar@{^(->}[d]\\
V\ar@{^{(}->}[r]^-{i_V}&C_{p+1,d}(\P^{n+1})\ar[r]^-{F_{tD}}\ar[ur]^{F_D}& C_{p+1,de}(\P^{n+1}),
}
}
\end{equation}
where $F_D:=F_{1D}$.

\begin{lemma}\label{lemma4.1}
For any integer $e\geq 1$,
the map
$$
\phi_e: C_{0,d}(\P^{n})\to C_{0,de}(\P^{n}), ~ \phi_e(c)=e\cdot c
$$
induces isomorphisms
$\phi_{e*}: \Ch_q(C_{0,d}(\P^{n}))_{\Q}\to \Ch_q(C_{0,de}(\P^{n}))_{\Q}$ for $q\leq d$.
\end{lemma}

\begin{proof}
First we note that $C_{0,d}(\P^n)\cong \sp^d(\P^n)$,
where $\sp^d(\P^n)$ denotes the $d$-th symmetric product of $\P^n$.
By Lemma \ref{lemma3.6},
we have
$$
\Ch_q(C_{0,d}(\P^{n}))_{\Q}\cong H_{2q}(C_{0,d}(\P^{n}))_{\Q}
$$ for all $d,q,n$.
By Lemma 9 and Remark 10 in \cite{Hu-2015},
we have
$$
 H_{2q}(C_{0,d}(\P^{n}))_{\Q}\cong  H_{2q}(C_{0,de}(\P^{n}))_{\Q}
$$
for $q\leq d$.
This completes the proof of the Lemma.
\end{proof}

\begin{remark}
We expect that
$\phi_{e*}: \Ch_q(C_{0,d}(\P^{n}))\to \Ch_q(C_{0,de}(\P^{n}))$
are isomorphisms for $q\leq d$
under the assumption of Lemma \ref{lemma4.1}.
\end{remark}

Note that there is a commutative diagram
\begin{equation}\label{eqn4.2}
\vcenter{
\xymatrix{C_{p,d}(\P^{n})\ar[r]^{\Phi_{p,d,n,e}}\ar[d]^{\Sigma}& C_{p,de}(\P^{n})\ar[d]^{\Sigma}\\
T_{p+1,d}(\P^{n+1})\ar[r]^{F_{tD}}& T_{p+1,de}(\P^{n+1}),
}
}
\end{equation}
where $\Phi_{p,d,n,e}$ ($\Phi_{0,d,n,e}=\phi_e$ in Lemma \ref{lemma4.1}) is the composed map
$$
\begin{array}{ccccc}
 C_{p,d}(\P^{n})&\to& C_{p,d}(\P^{n})\times\cdots \times C_{p,d}(\P^{n}) &\to& C_{p,de}(\P^{n})\\
 c&\mapsto& (c,...,c)&\mapsto& e\cdot c
\end{array}
$$
and $D\in \widetilde{Div_e}(\P^{n+2})$,
$F_{tD}$ is the restriction of
$F_{tD}: C_{p+1,d}(\P^{n+1})\to C_{p+1,de}(\P^{n+1})$
since its restriction on $T_{p+1,d}(\P^{n+1})$
is also in $T_{p+1,de}(\P^{n+1})$ (see \cite[Lemma 5.5]{Lawson1}).

%

\begin{proposition}\label{Prop4.3}
For integers $p,d,n$ such that $0\leq p\leq n$ and $d\geq 0$,
there is an integer $e_{p,d,n}\geq 1$ such that when $e\geq e_{p,d,n}$,
the induced map
$$
(\Phi_{p+1,d,n+1,e})_*: \Ch_q(C_{p+1,d}(\P^{n+1}))_{\Q}\to \Ch_q(C_{p+1,de}(\P^{n+1}))_{\Q}
$$
on Chow groups with rational coefficients  by
$$
\Phi_{p+1,d,n+1,e}:C_{p+1,d}(\P^{n+1})\to C_{p+1,de}(\P^{n+1}),c\mapsto e\cdot c
$$
are isomorphic for $0\leq q\leq d$.
Furthermore,
we have the following commutative diagram of isomorphisms
\begin{equation}\label{eqn4.3}
\vcenter{
\xymatrix{&\Ch_{q}(C_{p+1,d}(\P^{n+1}))_{\Q}\ar[r]^-{cl_{\Q}}_{\cong}\ar[d]^{(\Phi_{p+1,d,n+1,e})_*}_{\cong}
&H_{2q}(C_{p+1,d}(\P^{n+1}),\Q)\ar[d]^{(\Phi_{p+1,d,n+1,e})_*}_{\cong}\\
&\Ch_q(C_{p+1,de}(\P^{n+1}))_{\Q}\ar[r]^-{cl_{\Q}}_{\cong}&H_{2q}(C_{p+1,de}(\P^{n+1}),\Q)
}
}
\end{equation}
for $0\leq q\leq d$.
\end{proposition}

\begin{proof}
We first prove the isomorphic property of $(\Phi_{p+1,d,n+1,e})_*$ by induction on $p$.
The case that $p=-1$ follows from Lemma \ref{lemma4.1}.
Assume that
$\Phi_{p,d,n,e}: C_{p,d}(\P^{n})\to C_{p,de}(\P^{n})$
defined by $\Phi_{p,d,n,e}(c)=e\cdot c$ induces isomorphisms
$(\Phi_{p,d,n,e})_*: \Ch_q(C_{p,d}(\P^{n}))_{\Q}\to \Ch_q(C_{p,de}(\P^{n}))_{\Q}$
for $0\leq q\leq d$ and $e\geq e_{p,d,n}$.

By Equation \eqref{eqn4.1} and Equation \eqref{eqn4.2},
we have a commutative diagram
\begin{equation*}
\vcenter{
\xymatrix{&C_{p,d}(\P^{n})\ar[r]^{\Phi_{p,d,n,e}}\ar[d]^{\Sigma}& C_{p,de}(\P^{n})\ar[d]^{\Sigma}\\
&T_{p+1,d}(\P^{n+1})\ar[r]\ar@{^(->}@{^(->}[d]& T_{p+1,de}(\P^{n+1})\ar@{^(->}[d]\\
V\ar@{^{(}->}[r]^-{i_V}&C_{p+1,d}(\P^{n+1})\ar[r]^-{F_{0D}}\ar[ur]^{F_D}& C_{p+1,de}(\P^{n+1}).
}
}
\end{equation*}
Therefore, we get the following commutative diagram on Chow groups
\begin{equation}\label{eqn4.4}
\vcenter{
\xymatrix{&\Ch_q(C_{p,d}(\P^{n}))_{\Q}\ar[r]^{(\Phi_{p,d,n,e})_*}\ar[d]^{\Sigma_*}& \Ch_q(C_{p,de}(\P^{n}))_{\Q}\ar[d]^{\Sigma_*}\\
&\Ch_q(C_{p+1,d}(\P^{n+1}))_{\Q}\ar[r]^-{(F_{0D})_*}& \Ch_q(C_{p+1,de}(\P^{n+1}))_{\Q}.
}
}
\end{equation}
By induction, the upper horizontal map $(\Phi_{p,d,n,e})_*$ is an isomorphism.
Moreover, by Theorem \ref{Thm3.1},
$\Sigma_*: \Ch_q(\mC_{p}(\P^{n}))_{\Q}\to \Ch_q(\mC_{p+1}(\P^{n+1}))_{\Q}$
is an isomorphism for $0\leq q\leq d$,
then we get isomorphism $\Sigma_*: \Ch_q(C_{p,de}(\P^{n}))_{\Q}\to \Ch_q(C_{p+1,de}(\P^{n+1}))_{\Q}$
for $e$ sufficient large.
Hence the left vertical map $\Sigma_*$ is injective and the lower horizontal map
$(F_{0D})_*$ is surjective.

In fact, we can prove the surjective property of $(F_{0D})_*$ directly.
For $\alpha\in \Ch_q(C_{p+1,d}(\P^{n+1}))_{\Q}$,
it is linear combination of the class of (irreducible) subvariety of dimension $q$
on $C_{p+1,d}(\P^{n+1})$.
It is enough to consider that $\alpha$ is the cycle class of a subvariety $V$.
Let $\alpha\in \Ch_q(C_{p+1,d}(\P^{n+1}))_{\Q}$ be an element such that $\alpha=[V]$ is the cycle class
of a subvariety $V$ of dimension $q$ on $C_{p+1,d}(\P^{n+1})$.
Let $i_V:V\to C_{p+1,d}(\P^{n+1})$ be the inclusion.
To prove that $(F_{0D})_*$ is surjective,
we need to show that for large $e$,
there is a cycle $V'$ on $C_{p,de}(\P^{n})$
such that $\Sigma_*([V'])=(F_{0D})_*([V])\in \Ch_q(C_{p+1,de}(\P^{n+1}))$.
Indeed,
by Proposition \ref{Prop3.5},
$F_{0D}(V)$ is rationally equivalent to
$F_{D}(V)\subset T_{p+1,de}(\P^{n+1})$ on $C_{p+1,de}(\P^{n+1})$ for $e\geq e_{p+1,d,n+1}$.
Meanwhile,
$F_{D}(V)$  is rationally equivalent
to an algebraic cycle $\Sigma V'\subset \Sigma C_{p,de}(\P^{n})$ on $C_{p+1,de}(\P^{n+1})$.

Since $e\cdot [V]= [e\cdot V]=[F_{0D}(V)]=(F_{0D})_*([V])\in \Ch_q(C_{p+1,de}(\P^{n+1}))$,
we obtain that $(F_{0D})_*$ is injective for $e$ sufficient large (cf. Notes 5.1 in \cite{Lawson1} for homotopy groups).
Therefore, for sufficient large $e$ ,
$$
(\Phi_{p+1,d,n+1,e})_* =(F_{0D})_*:
\Ch_q(C_{p+1,d}(\P^{n+1}))_{\Q}\to \Ch_q(C_{p+1,de}(\P^{n+1}))_{\Q}
$$
are isomorphic for $0\leq q\leq d$.

Now we have shown  the left vertical map in the following commutative diagram
\begin{equation*}
\vcenter{
\xymatrix{&\Ch_{q}(C_{p+1,d}(\P^{n+1}))_{\Q}\ar[r]^-{cl_{\Q}}\ar[d]^{(\Phi_{p+1,d,n+1,e})_*}_{\cong}
&H_{2q}(C_{p+1,d}(\P^{n+1}),\Q)\ar[d]^{(\Phi_{p+1,d,n+1,e})_*}_{\cong}\\
&\Ch_q(C_{p+1,de}(\P^{n+1}))_{\Q}\ar[r]^-{cl_{\Q}}&H_{2q}(C_{p+1,de}(\P^{n+1}),\Q)
}
}
\end{equation*}
 is an isomorphism for $0\leq q\leq d$, while the right vertical map
 is also an isomorphism (see \cite{Hu-2015}).
 When $e\to \infty$,  we obtain the following diagram
\begin{equation}\label{eqn4.5}
\vcenter{
\xymatrix{&\Ch_{q}(C_{p+1,d}(\P^{n+1}))_{\Q}\ar[r]^-{cl_{\Q}}\ar[d]_{\cong}
&H_{2q}(C_{p+1,d}(\P^{n+1}),\Q)\ar[d]_{\cong}\\
&\Ch_q(\mC_{p+1}(\P^{n+1}))_{\Q}\ar[r]^-{cl_{\Q}}&H_{2q}(\mC_{p+1}(\P^{n+1}),\Q)
}
}
\end{equation}
for $0\leq q\leq d$. By Theorem \ref{Thm2.1}, the lower horizontal map in
Equation \eqref{eqn4.5} is an isomorphism. Therefore, the upper horizontal map is an isomorphism
for $0\leq q\leq d$.
This completes the proof of the Proposition.
\end{proof}

\begin{proof}[Proof of Theorem \ref{Thm2.2}]
By Proposition \ref{Prop4.3},
we have isomorphisms
$$
(\Phi_{p+1,d,n+1,e})_*: \Ch_q(C_{p+1,d}(\P^{n+1}))_{\Q}\to \Ch_q(C_{p+1,de}(\P^{n+1}))_{\Q}
$$
for $0\leq q\leq d$, $0\leq p\leq n$ and $e$ large.
Let $e\to \infty$, we get isomorphisms
$$
\Ch_q(C_{p+1,d}(\P^{n+1}))_{\Q}\to \Ch_q(\mC_{p+1}(\P^{n+1}))_{\Q}.
$$
for $0\leq q\leq d$, $0\leq p\leq n$.
Sine $i:C_{p+1,d}(\P^{n+1}))\to C_{p+1,d+1}(\P^{n+1}))$ commutes with  $\Phi_{p+1,*,n+1,e}$,
we get the following commutative diagram of Chow groups
\begin{equation*}
\xymatrix{&\Ch_q(C_{p,d}(\P^{n}))_{\Q}\ar[r]^{\cong}\ar[d]^{i_*}& \Ch_q(\mC_{p}(\P^{n}))_{\Q}\\
&\Ch_q(C_{p,d+1}(\P^{n}))_{\Q}\ar[ru]^-{\cong}& .
}
\end{equation*}
for $0\leq q\leq d$ and $1\leq p\leq n$.
This shows that
$$
i_*:\Ch_q(C_{p,d}(\P^{n}))_{\Q}\to \Ch_q(C_{p,d+1}(\P^n))_{\Q}
$$
are isomorphisms for $0\leq q\leq d$ and $1\leq p\leq n$.

For $p=0$, $i_*:\Ch_q(C_{0,d}(\P^{n}))_{\Q}\to \Ch_q(C_{0,d+1}(\P^n))_{\Q}$
is an isomorphism for $0\leq q\leq d$
since by Lemma \ref{lemma3.6},
the cycle class map is a natural isomorphism and
$i_*:H_{2q}(C_{0,d}(\P^{n}))_{\Q}\to H_{2q}(C_{0,d+1}(\P^n),\Q)$
is also an isomorphism for $0\leq q\leq d$.
This completes the proof of the theorem.
\end{proof}

\begin{proof}[Proof of Theorem \ref{Thm2.3}]
In the commutative diagram \eqref{eqn4.4},
we obtain that the left vertical map
$$
\Sigma_*:\Ch_q(C_{p,d}(\P^{n}))_{\Q}\to \Ch_q(C_{p+1,d}(\P^{n+1}))_{\Q}
$$
are isomorphic for $0\leq q\leq d$
from the facts that all other three maps are isomorphisms,
as shown in the proof of Proposition \ref{Prop4.3}.
This completes the proof of Theorem \ref{Thm2.3}.
\end{proof}

\begin{proof}[Proof of Corollary \ref{cor2.4} ]
It follows immediately
from the upper horizontal isomorphism in Equation \eqref{eqn4.3} in Proposition \ref{Prop4.3}.
\end{proof}


\section{Application to cycles of fixed degree}
In this section, we apply our result to the space of all algebraic cycles of a fixed degree.
Let $d\geq 1$ be a fixed integer.
Consider the spaces
\begin{equation}\label{eqn5.1}
\D(d):=\lim_{p,q\to \infty} C_{p,d}(\P^{p+q})
\end{equation}
of cycles of a fixed degree (with arbitrary dimension and codimension),
as introduced in \cite{LM},
where  the limit for $p$ is given by suspension $\Sigma:C_{p,d}(\P^n)\to C_{p+1,d}(\P^{n+1}) $
and the limit for $q$ is induced by the inclusion $\P^{p+q}\subset \P^{p+q+1}$,
i.e. a $p$-cycle in $C_{p,d}(\P^{p+q})$ is viewed as a $p$-cycle in $C_{p,d}(\P^{p+q+1})$.

Then there is a filtration (see \cite[\S7]{LM})
$$
BU=\D(1)\subset\D(2)\subset\cdots\subset \D(\infty),
$$
where $BU=\lim\limits_{q\to \infty}BU_q$,
$\D(\infty)$ is homotopic to $K(even,\Z)=\prod_{i=1}^{\infty} K(2i,\Z)$
 (the weak product of Eilenberg-Maclane spaces).

The inclusion map $\D(d)\subset \D(\infty)$ induces maps on Chow groups.
It is an easy corollary of Theorem \ref{Thm2.2}
that $\D(1)\subset \D(\infty)$ induces an injective map on Chow groups
and surjective on Chow group with rational coefficients.

The following question is a natural analogous of its homological version.
\begin{question}\label{ques1}
Is $\Ch_*(\D(d))\to \Ch_*(\D(\infty))$ injective for all $d\geq 1$?
\end{question}

The homotopic version of the above question was proposed in \cite{LM},
which was answered negatively in \cite{Hu-Deg2} for $d=2$
through an explicit calculation of homology groups of $\D(2)$.
It is too optimistic to get a positive answer to Question \ref{ques1}.

However, we have the following result.
\begin{theorem}
The induced map on the Chow groups
$$
i_*:\Ch_q(\D(d))_{\Q}\to \Ch_q(\D(\infty))_{\Q}
$$
from the inclusion
$i:\D(d)\subset \D(\infty)$
is an isomorphism for $q\leq d$.
\end{theorem}
\begin{proof}
The result follows from Theorem \ref{Thm2.2} by taking the limit
$p,n\to \infty$.
\end{proof}

\begin{remark}
In general,
the Chow group of a complex algebraic variety is not a homotopy invariant.
Hence the direct limit $\lim\limits_{n\to \infty}\Ch_*(V^n)$
for a direct system $i_n:V^n\to V^{n+1}$ of morphisms
$i_n$ between algebraic varieties is not a homotopy invariant, either.
$D(\infty)$ is just one of the models whose homotopy type is $K(even,\Z)$.
In this model of  $K(even,\Z)$,
the Chow groups with rational coefficients coincide with the corresponding rational homology group.

It is not hard to find  another choice of the model of $K(even,\Z)$ whose Chow groups is
different from the corresponding rational homology group.
This implies that the Chow group of Eilenberg-MacLane spaces, in general, is not well defined.
Note that Eilenberg-MacLane spaces are classifying spaces of topological abelian groups.
This is compared to the definition of Chow group of the classifying space of any linear algebraic group by  Totaro (see \cite{Totaro}).
\end{remark}

\end{document}